\documentclass{amsart}
\usepackage{amssymb}
\usepackage[margin=1.1in]{geometry}
\usepackage{mathtools}

\theoremstyle{plain} %=default

\newtheorem{theorem}{Theorem}%[section]

\newtheorem{lemma}[theorem]{Lemma}
\newtheorem{corollary}[theorem]{Corollary}

\theoremstyle{definition}

\theoremstyle{remark}
\newtheorem{remark}{Remark}

\newcommand{\R}{\ensuremath{\mathbb{R}}}

\newcommand{\st}{\ensuremath{\::\:}}

\newcommand{\re}{\mbox{\textup{r.e.}}}

\begin{document}

\title{A note on Cauchy integrability}
\author{Scott Schneider}
\address{Mathematics Department\\
University of Michigan\\
Ann Arbor, MI 48109--1043, U.S.A.}
\email{sschnei@umich.edu}

\begin{abstract} We show that for any bounded function $f:[a,b]\rightarrow\R$ and $\epsilon>0$ there is a partition $P$ of $[a,b]$ with respect to which the Riemann sum of $f$ using right endpoints is within $\epsilon$ of the upper Darboux sum of $f$. This leads to an elementary proof of the theorem of Gillespie \cite{G} showing that Cauchy's and Riemann's definitions of integrability coincide.
\end{abstract}

\maketitle

\section{Introduction}

A natural reaction to Riemann's definition of integrability is to wonder how far it can be simplified, and many of the resulting questions can reasonably be asked of undergraduates in a first course on real analysis.  For instance, it is relatively easy to show that altering his definition by allowing only regular partitions or by allowing only endpoints as sample points does not enlarge the class of integrable functions, but that making both restrictions at once does. On the other hand it appears to be less widely known, and to be somewhat trickier to show using elementary methods, that for bounded functions allowing only \emph{left} endpoints as sample points does not enlarge the class of integrable functions.  This was first shown in 1915 by Gillespie \cite{G} using oscillation and measure, and was later reproved by Kristensen, Poulsen, and Reich \cite{KPR} with a more elementary argument that could be presented to undergraduates or assigned to them as independent study. In this note we give another elementary proof of Gillespie's result using a lemma that may be of independent interest. Our proof avoids measure theory and requires only basic facts about the Darboux integral.

What we call \emph{left-endpoint Riemann integrability} below is in fact the definition of integrability that Cauchy \cite{C} originally gave for (continuous) functions in 1823.  The papers \cite{G} and \cite{KPR} use left endpoints, following Cauchy, but here we use right endpoints instead; since $f(x)$ could always be replaced by $f(-x)$, the difference is inconsequential.

\section{Notation}

Let $f$ be a function on the interval $[a,b]$. We write 
\[
R(f,P,\{x_k^*\}) \ := \ \sum_{k=1}^n f(x_k^*)(x_k-x_{k-1})
\]
for the Riemann sum of $f$ on $[a,b]$ using the partition $P=\{x_k\}_{k=0}^n$ of $[a,b]$ and sample points $x_k^*\in [x_{k-1},x_k]$. We write $R(f,P,\re)$ for the Riemann sum that uses right endpoints $x_k^*=x_k$ as sample points in each subinterval. The function $f$ is \emph{Riemann integrable} on $[a,b]$ if there is $L\in\mathbb R$ such that for every $\epsilon>0$ there is $\delta>0$ such that for all partitions $P=\{x_k\}_{k=0}^n$ of $[a,b]$ with mesh $\|P\|:=\max\{x_k-x_{k-1}\st 1\leq k\leq n\}<\delta$ and for all choices $\{x_k^*\}$ of sample points relative to $P$, $|R(f,P,\{x_k^*\})-L|<\epsilon$. We say that $f$ is \emph{right-endpoint Riemann integrable} if the same holds when restricting to right-endpoint Riemann sums, and define \emph{left-endpoint Riemann  integrability} analogously.

If $f$ is bounded, we write 
\[
\begin{array}{rcl}
U(f,P) & := & \displaystyle{\sum_{k=1}^n(\sup(f, [x_{k-1},x_k]))\cdot (x_k-x_{k-1})}, \\
L(f,P) & := & \displaystyle{\sum_{k=1}^n(\inf(f, [x_{k-1},x_k]))\cdot (x_k-x_{k-1})}
\end{array}
\]
%\[
%U(f,P) \ := \ \sum_{k=1}^n(\sup(f, [a,b]))\cdot (x_k-x_{k-1}), \quad L(f,P) \ := \ \sum_{k=1}^n(\inf(f, [a,b]))\cdot (x_k-x_{k-1})
%\]
for the upper and lower Darboux sums of $f$ on $[a,b]$ relative to the partition $P$, and 
\[
\begin{array}{rcl}
U(f) & := & \inf\{U(f,P)\st \mbox{$P$ is a partition of $[a,b]$}\}, \\
L(f) & := & \sup\{L(f,P)\st \mbox{$P$ is a partition of $[a,b]$}\}
\end{array}
\]
for the upper and lower Darboux integrals of $f$ on $[a,b]$. The function $f$ is Darboux integrable if $U(f)=L(f)$ by definition, and Riemann integrable if and only if Darboux integrable by the well-known correspondence.%Whenever it is necessary, we will include reference to the interval $[a,b]$ explicitly in our notation.

%The function $f:[a,b]\to\R$ is \emph{Riemann integrable} if if there is $L\in\mathbb R$ such that for every $\epsilon>0$ there is $\delta>0$ such that for all partitions $P=\{x_k\}_{k=0}^n$ of $[a,b]$ with mesh $\|P\|<\delta$ and for all choices $\{x_k^*\}$ of sample points relative to $P$, $|R(f,P,\re)-L|<\epsilon$. We say that $f$ is \emph{right-endpoint Riemann integrable} if the same holds when restricting to right-endpoint Riemann sums.  Left-endpoint integrability is defined analogously. We will show that if $f$ is bounded and right-endpoint Riemann integrable then $f$ is Riemann integrable.

%The function $f:[a,b]\to\mathbb R$ is said to be \emph{right-endpoint Riemann integrable} if there is $L\in\mathbb R$ such that for every $\epsilon>0$ there is $\delta>0$ such that for all partitions $P=\{x_k\}_{k=0}^n$ of $[a,b]$ with mesh $\|P\|<\delta$, $|R(f,P,\re)-L|<\epsilon$. Left-endpoint integrability is defined analogously. We will show that if $f$ is bounded and right-endpoint Riemann integrable then $f$ is Riemann integrable.

\section{The proof}

\begin{lemma} If $f:[a,b]\to\mathbb R$ is bounded, then for every $\epsilon>0$ there is a partition $Q$ of $[a,b]$ such that $U(f,Q)-R(f,Q,\re)<\epsilon$.\label{lem} \end{lemma}

\begin{proof} Fix $B$ such that $\sup(f,[a,b])-\inf(f,[a,b])\leq B$, and let $\epsilon>0$. Define $g:[a,b]\to\mathbb R$ by $g(x)=\sup(f,[x,b])$, so that $g$ is decreasing, and hence Riemann integrable, on $[a,b]$, say with value $L=\int_a^bg$. Fix $\delta_1>0$ such that for any partition $P$ of $[a,b]$ with $\|P\|<\delta_1$ and for any choice $x_k^\ast$ of sample points relative to $P$, $|R(g,P,\{x_k^\ast\})-L|<\epsilon$. Let $\delta_2=\frac{\epsilon}{B}$, and set $\delta=\min\{\delta_1,\delta_2\}$. Let $P=\{x_k\}_{k=0}^n$ be a partition of $[a,b]$ with $\|P\|<\frac{\delta}{2}$. We will use $P$ in order to obtain the desired partition $Q$. Let 
\[
\begin{array}{lll}
C & = & \{1\leq k\leq n\st g(x_{k-1})=g(x_k)\}; \\
D & = & \{1\leq k\leq n\st g(x_{k-1})>g(x_k)\}; \\ % \{1,\ldots,n\}\setminus C; \\
D' & = & \{k\in D\setminus\{n\}\st k+1\not\in D\}.
\end{array}
\]
For each $k\in D'$ let $z_k=\inf\{x\in [x_{k-1},x_k]\st g(x)=g(x_k)\}$. Let
\[
\begin{array}{lll}
D_0' & = & \{k\in D'\st g(z_k)=g(x_k)\}; \\
D_1' & = & \{k\in D'\st g(z_k)>g(x_k)\}.
\end{array}
\]
For each $k\in D$, choose $y_k\in [x_{k-1},x_k)$ such that $|f(y_k)-g(x_{k-1})|<\frac{\epsilon}{b-a}$, where additionally if $k\in D_0'$ then $y_k<z_k$, and if $k\in D_1'$ then $y_k\leq z_k$. Also for each $k\in D'$ choose an additional point in $[x_{k-1},x_k]$ as follows. If $k\in D_0'$, choose $u_k\in (y_k,z_k)$ such that $|u_k-z_k|<\frac{\epsilon}{B|D'|}$ and $|f(u_k)-g(u_k)|<\frac{\epsilon}{b-a}$; if $k\in D_1'$, choose $v_k\in (z_k,x_k)$ such that $|v_k-z_k|<\frac{\epsilon}{B|D'|}$. Now let
\[
\begin{array}{lll}
Q_0 & = & \{a,b\}; \\
Q_1 & = & \{y_k\st k\in D\,\wedge\,k-1\not\in D\}; \\
Q_2 & = & \{z_k\st k\in D_0'\}\,\cup\,\{v_k\st k\in D_1'\}; \\
Q_3 & = & \{z_k\st k\in D_1'\,\wedge\,z_k\ne y_k\}\,\cup\,\{u_k\st k\in D_0'\}\,\cup\,\{y_k\st k,k-1\in D\}; \\
Q & = & Q_0\cup Q_1\cup Q_2\cup Q_3.
\end{array}
\]
We show that $U(f,Q)-R(f,Q,\re)<6\epsilon$.

For each $q\in Q$, let $q'=q$ if $q=a$ and otherwise let $q'$ be the largest element of $Q$ strictly less than $q$. For $q\in Q$ write 
\[
E_q \ := \ \big(\sup(f,[q',q])-f(q)\big)(q-q'),
\]
so that
\[
U(f,Q)-R(f,Q,\re) \ = \ \sum_{q\in Q}E_q \ \leq \ \sum_{i=0}^3\left(\sum_{q\in Q_i}E_q\right).
\]
We now bound the sums $\displaystyle{\sum_{q\in Q_i}E_q}$ for each $i\leq 3$.

First note that if $g(x_{n-1})=g(x_n)$ then $\sup(f,[b',b])=f(b)$ and hence $E_b=0$. On the other hand if $g(x_{n-1})>g(x_n)$ then $n\in D$, so $b'\in [x_{n-1},b)$ and hence $E_b<\frac{B\delta}{2}$. Since $E_a=0$, we get
\[
\sum_{q\in Q_0}E_q \ < \ \frac{B\delta}{2} \ < \ \epsilon.
\]
Next notice that if $q\in Q_1$ then $q'=a$ or $q'\in Q_2$. In either case $\sup(f,[q',q])-f(q)<\frac{\epsilon}{b-a}$, and therefore
\[
\sum_{q\in Q_1}E_q \ < \ \epsilon.
\]
If $q\in Q_2$ then $q-q'<\frac{\epsilon}{B|D'|}$, and since $|Q_2|\leq |D'|$ this implies
\[
\sum_{q\in Q_2}E_q \ < \ B|D'|\cdot\frac{\epsilon}{B|D'|} \ = \ \epsilon.
\]
Finally, suppose $q\in Q_3$. If $q=z_k\ne y_k$ where $k\in D_1'$, then $g(q)=f(q)$ and $q'=y_k$ so $q-q'<\frac{\delta}{2}$. If $q=u_k$ where $k\in D_0'$ then $|f(q)-g(q)|<\frac{\epsilon}{b-a}$ and $q'=y_k$ so again $q-q'<\frac{\delta}{2}$. Finally if $q=y_k$ where $k,k-1\in D$, then
\[
g(x_{k-1})-\frac{\epsilon}{b-a} \ \leq \ f(q) \ \leq \ g(q) \ \leq \ g(x_{k-1})
\]
and $q'=y_{k-1}$ so $q-q'<\delta$. Thus in all cases $q-q'<\delta\leq\delta_1$ and $|f(q)-g(q)|<\frac{\epsilon}{b-a}$. Let now $Q'$ be any partition of $[a,b]$ that extends $Q_3\cup\{q'\st q\in Q_3\}$, satisfies $\|Q'\|<\delta_1$, and has no points in the intervals $(q',q)$ for $q\in Q_3$. Then
\[
\begin{array}{lllll}
\displaystyle{\sum_{q\in Q_3}E_q} & = & \displaystyle{\sum_{q\in Q_3}\big(\sup(f,[q',q])-f(q)\big)(q-q')} & & \medskip \\
& \leq & \displaystyle{\sum_{q\in Q_3}\big(g(q')-f(q)\big)(q-q')} & & \medskip \\
& < & \displaystyle{\epsilon+\sum_{q\in Q_3}\big(g(q')-g(q)\big)(q-q')} & & \medskip \\
& \leq & \displaystyle{\epsilon+\sum_{q\in Q'}\big(g(q')-g(q)\big)(q-q')} & & \medskip \\
& = & \displaystyle{\epsilon+\sum_{q\in Q'} g(q')(q-q') - \sum_{q\in Q'}g(q)(q-q')} & < & 3\epsilon.
\end{array}
\]
Combining these estimates we have
\[
U(f,Q)-R(f,Q,\re) \ = \ \sum_{q\in Q}E_q \ < \ 6\epsilon,
\]
and since $\epsilon$ was arbitrary this completes the proof. \end{proof}

\begin{corollary} If $f:[a,b]\to\mathbb R$ is bounded, then for every $\epsilon>0$ there is a partition $Q$ of $[a,b]$ such that $R(f,Q,\re)-L(f,Q)<\epsilon$.\label{cor} \end{corollary}

\begin{proof} Apply Lemma \ref{lem} to $-f$. \end{proof}

\begin{theorem}[\cite{G},\cite{D},\cite{KPR}] Let $f:[a,b]\to\mathbb R$ be a bounded function. If $f$ is right-endpoint Riemann integrable on $[a,b]$, then $f$ is Riemann integrable on $[a,b]$.\label{prop} \end{theorem}

\begin{proof} Let $L$ be the limiting value of the right-endpoint Riemann sums of $f$ on $[a,b]$. Let $\epsilon>0$ be arbitrary, and fix $\delta>0$ such that for all partitions $P$ of $[a,b]$ with $\|P\|<\delta$, $|R(f,P,\re)-L|<\epsilon$. Let $P=\{x_k\}_{k=0}^n$ be a partition of $[a,b]$ with $\|P\|<\delta$. Using Lemma \ref{lem} and Corollary \ref{cor}, for each $1\leq k\leq n$ choose partitions $Q_k^U$ and $Q_k^L$ of $[x_{k-1},x_k]$ such that on $[x_{k-1},x_k]$,
\[
U(f,Q_k^U)-R(f,Q_k^U,\re) \ < \ \frac{\epsilon}{n}\quad\mbox{and}\quad R(f,Q_k^L,\re)-L(f,Q_k^L) \ < \ \frac{\epsilon}{n}.
\]
Let $Q^U=\cup_kQ^U_k$ and $Q^L=\cup_kQ^L_k$. Then
\[
U(f,Q^U)-R(f,Q^U,\re) \ < \ \epsilon\quad\mbox{and}\quad R(f,Q^L,\re)-L(f,Q^L) \ < \ \epsilon.
\]
Since $\|Q^U\|,\|Q^L\|<\delta$, we have
\[
|R(f,Q^U,\re)-L| \ < \ \epsilon\quad\mbox{and}\quad |R(f,Q^L,\re)-L| \ < \ \epsilon.
\]
Thus
\[
U(f,Q^U)-L(f,Q^L) \ < \ 4\epsilon,
\]
which implies $U(f)-L(f)<4\epsilon$. Since $\epsilon$ was arbitrary, it follows that $f$ is Darboux integrable on $[a,b]$, and therefore Riemann integrable on $[a,b]$. \end{proof}

%\begin{remark} One may define \emph{left-endpoint Riemann integrability} in the obvious fashion. By replacing $f(x)$ with the function $f(-x)$ and applying Proposition \ref{prop}, we see that if $f:[a,b]\to\mathbb R$ is bounded and left-endpoint Riemann integrable then $f$ is Riemann integrable. Left-endpoint Riemann integrability is in fact the definition of integrability that Cauchy \cite{C} originally gave for continuous functions in 1823. \end{remark}

\begin{remark} The assumption of boundedness in Proposition \ref{prop} is necessary, but in a benign way.  For instance, one easily checks that the function 
\[
f(x)=\begin{cases} x^{-1/2} & \mbox{if $x>0$} \\ 0 & \mbox{if $x=0$} \end{cases}
\]
is right-endpoint Riemann integrable on $[0,1]$ even though it is unbounded. However, it is easy to show that if $f:[a,b]\to\mathbb R$ is right-endpoint Riemann integrable with limit $L$, then $f$ is Riemann integrable (in particular, bounded) on $[c,b]$ for any $a<c<b$, and  %$\displaystyle{\lim_{c\to a^+}\int_c^b\!f(x)\,dx=L}$.
\[
\lim_{c\to a^+}\int_c^b\!f(x)\,dx = L.
\]
\end{remark}

%\begin{remark} The authors of \cite{KPR} actually prove the following stronger result using oscillation and measure. Let $f$ be a bounded function on $[a,b]$ and $\psi$ a function with domain $\{(x,y)\st a\leq x\leq y\leq b\}$ such that for all $x,y$, $x\leq\psi(x,y)\leq y$. Call $f$ \emph{$\psi$-integrable} if the Riemann sums $\sum f(\psi(x_{k-1},x_k))(x_k-x_{k-1})$ converge to a common limit. Then if $\psi$ satisfies the conclusion of the intermediate value theorem separately in each variable, $f$ is Riemann integrable if $\psi$-integrable (see \cite[Theorem 2]{KPR}).
%\end{remark}

\begin{remark} Dantoni \cite{D} and the authors of \cite{KPR} independently proved stronger results using oscillation and measure. Let $f$ be a bounded function on $[a,b]$ and $\psi$ a function with domain $\{(x,y)\st a\leq x\leq y\leq b\}$ such that $x\leq\psi(x,y)\leq y$ for all $x,y$. Call $f$ \emph{$\psi$-integrable} if the Riemann sums $\sum f(\psi(x_{k-1},x_k))(x_k-x_{k-1})$ converge to a common limit.  In \cite{KPR} the authors establish sufficient conditions on $\psi$, and in \cite{D} Dantoni establishes necessary and sufficient conditions on $\psi$, for every bounded function $f$ on $[a,b]$ to be $\psi$-integrable; in particular both papers show that continuity of $\psi$ suffices, generalizing Gillespie's result.
\end{remark}

\end{document}